\documentclass[11pt]{amsart}
\usepackage{amsmath,amsthm,amsfonts,amssymb,MnSymbol}

\newcommand{\C}{\mathbb{C}}
\newcommand{\Z}{\mathbb{Z}}
\newcommand{\R}{\mathbb{R}}
\newcommand{\Q}{\mathbb{Q}}
\newcommand{\Qb}{\bar{\mathbb{Q}}}
\newcommand{\F}{\mathbb{F}}
\newcommand{\A}{\mathbb{A}}
\newcommand{\N}{\mathbb{N}}
\newcommand{\PP}{\mathbb{P}}

\newcommand{\X}{\mathfrak{X}}

\newcommand{\Gal}{\operatorname{Gal}}

\newcommand{\pre}{\operatorname{pre}}

\newtheorem{thm}{Theorem}
\newtheorem{quest}[thm]{Question}

\newtheorem{cor}[thm]{Corollary}

\newtheorem{prop}[thm]{Proposition}

\newtheorem{lem}[thm]{Lemma}

\newtheorem*{claim*}{Claim}
\begin{document}

\title[Finitary Hasse Principle]{A Finitary Hasse Principle for Diagonal Curves}
\author{Jean Bourgain and Michael Larsen}
\date{\today}
\address{School of Mathematics,
Institute for Advanced Study,
Einstein Drive, Princeton, NJ 08540, USA}
\email{bourgain@math.ias.edu}
  \address{Department of Mathematics, Indiana University, Bloomington,
Indiana 47405, USA} \email{mjlarsen@indiana.edu}

\thanks{Jean Bourgain thanks the Berkeley math department for its hospitality.
Michael Larsen thanks the MSRI for its hospitality and  also  wants to acknowledge support from  NSF grant DMS-1101424
and the Simons Foundation.}

\maketitle
\begin{abstract}
We prove a Hasse principle for solving equations of the form $ax+by+cz=0$ where $x,y,z$ belong to a given finite index subgroup of $\Q^\times$.
From this we deduce a Hasse principle for diagonal curves over subfields of $\bar\Q$ with finitely generated Galois group.
\end{abstract}

\section{Introduction}

Let $a,b,c$ be non-zero rational numbers and $n\ge 2$ an integer.  Let $X$ denote the projective curve
$ax^n+by^n+cz^n=0$.  For $n=2$, the following are equivalent:

\begin{enumerate}
\item $X(\Q_p)\neq \emptyset$ for all $p$ and $X(\R)\neq\emptyset$.
\item $X(\Q)\neq\emptyset$.
\item $X(\Q)$ is infinite.
\end{enumerate}

The equivalence of (1) and (2) is the Hasse-Minkowski theorem for conics over $\Q$, while the equivalence of (2) and (3) follows from stereographic projection.  For $n>2$, neither equivalence holds
in general.  Already for $n=3$, the Tate-Shafarevich group gives an obstruction to (1)$\Rightarrow$(2); for instance, Selmer showed that
$3x^3+4y^3+5z^3=0$ has local solutions for all places of $\Q$ but no global solution \cite[p.~8]{K}.
For $a=b=-c=1$, Fermat's Last Theorem shows that (2) does not imply (3) for any $n\ge 3$.

We fix once and for all an algebraic closure $\Qb$ of $\Q$.
We can view elements of $X(\Q)$ as elements of $X(\Qb)$ which are invariant under the action of
$G_{\Q} := \Gal(\Qb/\Q)$.  As $G_{\Q}$ is not finitely generated, this can be regarded as an infinitary condition.  It turns out that if we replace invariance under $G_{\Q}$ by any finite collection of invariance conditions, the equivalence of conditions (1)--(3) as above holds for all $n$ and all $a,b,c$.

Let $\Sigma\subset G_{\Q}$ be any finite subset.  Let 
$$K_\Sigma := \{x\in \Qb\mid \sigma(x)=x\ \forall \sigma\in \Sigma\}$$ 
denote the field of  invariants of the closed subgroup $\langle\Sigma\rangle$ generated by $\Sigma$.
A subfield $K$ of $\Qb$ is of this form if and only its absolute Galois group $G_K$ is (topologically) finitely generated.  We prove the following theorem:

\begin{thm}
\label{Hasse-Minkowski}
Given $a,b,c\in \Q^\times$ and $n$ a positive integer, the following conditions on the projective curve $X: ax^n+by^n+cz^n=0$ are equivalent:
\begin{enumerate}
\item $X(\Q_p)\neq \emptyset$ for all $p$ and $X(\R)\neq\emptyset$.
\item $X(K)\neq\emptyset$ for all $K\subset \Qb$ with $G_K$ finitely generated.
\item $|X(K)|=\infty$ for all $K\subset \Qb$ with $G_K$ finitely generated.
\end{enumerate}
\end{thm}

One can prove that (2) implies (3) in greater generality:

\begin{thm}
\label{Infinity-Zero}
If $K$ is a field in characteristic zero such that $G_K$ is finitely generated, then $X(K)$ non-empty implies $|X(K)|=\infty$.
\end{thm}

The proof of Theorem~\ref{Infinity-Zero} is purely combinatorial, following the strategy of \cite{IL}.

The proof that (1) implies (2) is more difficult and depends on the following Hasse principle, unusual in that we need to
consider finite combinations of local
conditions:

\begin{thm}
\label{Linear}
Let $G$ denote a finite index subgroup of $\Q^\times$, and let $a,b,c$ belong to $\Q^\times$.
For every set $S$ of places of $\Q$, we define $\Q_S := \prod_{v\in S} \Q_v$ and let $G_S$ denote the closure of $G$ in $\Q_S^\times$.
Then
\begin{equation}
\label{abc}
ax+by+cz=0
\end{equation}
has a solution for $x,y,z\in G$ if and only if the same equation has a solution in $G_S$ for all finite $S$.
\end{thm}

It is a striking fact that it does not suffice to check solvability in $G_S$ for singleton sets $S=\{v\}$---see Proposition~\ref{Doubles} below.
We remark also that solving (\ref{abc}) in $G$ is equivalent to solving it in any coset of $G$.  Richard Rado \cite{R} considered
which systems of homogeneous linear equations have the property that for every finite partition of $\N$, the system can be solved with all variables belonging
to a single part of the partition.  In the case of a single equation (\ref{abc}), the system satisfies this property if and only if $a+b=0$, $b+c=0$, $c+a=0$, or $a+b+c=0$.
In these special cases, therefore, Theorem~\ref{Linear} follows directly from Rado's theorem.  This corresponds to the fact that Theorem~\ref{Infinity-Zero} can be deduced from Ramsey theory, while the general case of Theorem~\ref{Hasse-Minkowski} requires the circle method.

\section{The Circle Method and Multiplicative Functions on $\Q$}
In this section, we apply the circle method to prove Theorem~\ref{Linear}.  We begin with some preliminary lemmas.

We fix a finite index subgroup $G\subset \Q^\times$ and non-zero $a,b,c\in\Q$
such that $ax+by+cz=0$ has a solution in $G_S$ for all finite sets $S$ of places of $\Q$.
We can freely replace $a$, $b$, or $c$ by any element in its $G$-coset, and we are free to multiply all three of them by a common non-zero rational number.

\begin{lem}
\label{devissage}
For all integers $D>0$, 
there exist elements $x,y,z\in G$ and $w\in \Q^\times$ such that $a' := wax$, $b' := wby$, $c' := wcz$ satisfy the following properties:
\begin{enumerate}
\item[(a)] $\min(a',b',c')<0$,
\item[(b)] $\max(a',b',c')>0$,
\item[(c)] $a' + b' + c' \equiv 0 \pmod D,$
\item[(d)] $a'$, $b'$ and $c'$ are pairwise relatively prime,
\item[(e)] $a',b',c'\in \Z$,
\item[(f)] $a'b'c'$ is even.
\end{enumerate}
\end{lem}

\begin{proof}
The proof consists of a series of steps in which we replace $a$, $b$, and $c$ by $wax$, $wby$, and $wcz$ respectively,
with the goal that at the end of the process, the resulting triple $a,b,c$ satisfies properties (a)--(f).

Let $\PP$ denote the set of all prime numbers
and $\PP_0$  the 
set of prime divisors of $D$.
Let $Q := \Q^\times/G$, and define 
$$\phi = (\phi_1,\phi_2)\colon \PP\setminus \PP_0\to Q\times (\Z/D\Z)^\times,$$ 
where $\phi_1$
denotes the restriction of the quotient map $\Q^\times\to Q$ to $\PP\setminus \PP_0$ and $\phi_2$ denotes the 
restriction of $\Z\to \Z/D\Z$ to $\PP\setminus \PP_0$.
Let $S$ be the  
union of all finite subsets of the form $\PP\cup\{\infty\}\setminus \phi^{-1}(Q')$
where $Q'$ is a subgroup of $Q\times (\Z/D\Z)^\times$.  Thus $S$ is finite, and if $p\not\in S$ and $M$ is a given integer, then there exists a product $m$
of primes $>M$ such that $\phi(pm)=1$.

By hypothesis, equation (\ref{abc}) has a solution $(x_S,y_S,z_S)$ in $G_S$.
Let $x_v$ denote the $v$-component of $x_S$ for $v\in S$ and likewise for $y_v,z_v$.  
As $a x_\infty + b y_\infty + c z_\infty = 0$, it follows that replacing $a,b,c$ by 
$a x,b y,c z$, where $x,y,z$ are sufficiently close to $x_\infty$, $y_\infty$, $z_\infty$,
the resulting triple satisfies properties (a) and (b).

Choose $k$ to be a positive integer larger than 
$$\max_{p\in S} \max(v_p(x_p),v_p(y_p),v_p(z_p))+v_p(D)$$
and choose $x,y,z\in G$ such that for all $p\in S\setminus\{\infty\}$,
$$v_p(x_p - x), v_p(y_p-y), v_p(z_p-z) > k,$$
and $ax$, $by$, and $cz$ are neither all positive nor all negative.  
Multiplying each of these by 
$$w := \prod_{p\in S} p^{-\min(v_p(ax),v_p(by), v_p(cz))},$$
we obtain $wax$, $wby$, $wcz$ which add to $0$ (mod $D$) and to zero (mod $p$) for each $p\in S$.
Moreover, for each $p$, all three belong to $\Z_p$, and at least one of the three belongs to $\Z_p^\times$; as they sum to zero (mod $p$), at least two are units.
Replacing $a,b,c$ by $wax,wby,wcz$, the resulting triple now satisfies properties (a)--(c), and at most one of $v_p(a),v_p(b),v_p(c)$ is positive for $p\in S$.

If $a$, $b$, or $c$ fails to be $p$-integral for some $p\not\in S$, by definition of $S$, there exists $m\in \N$ such that $pm\in G$, $pm\equiv 1$ (mod $D$), 
and all prime factors 
of $m$ are as large as we may wish.  In particular, we may assume that for each prime factor $q$ of $m$, $q\neq p$, $q\not\in S$, and $v_q(a)=v_q(b)=v_q(c)=0$.
Multiplying  by $pm$ eliminates a factor of $p$ from the denominator of
the desired element, $a$, $b$, or $c$, without changing the residue class (mod $D$) or the sign of the given element
or introducing a common prime factor of any two elements of the set.
Continuing this process as long as necessary, we can assume that
the resulting elements satisfy (a)--(e).
 If $a$, $b$, and $c$ are all odd, then $D$ is odd as well, so $2^k\equiv 1$ (mod $D$) for some positive integer $k$ divisible by $|Q|$; replacing
$a$ by $2^ka$, we obtain a new triple $a,b,c$ satisfying properties (a)--(f).
\end{proof}

\begin{lem}
\label{product}
Let $D$ be a positive integer.
Let $a,b,c$ be integers satisfying conditions (a)--(f).  There exists a constant $\epsilon > 0$ and for every prime $p$ a constant $d_p > \max(1,1-3/p)$
such that for every finite set $S$ of primes not dividing $D$, the number of solutions of (\ref{abc}) in $x,y,z\in (1+D\Z )\cap [0,N]$ such that $xyz$ is not divisible by any prime in $S$ is
at least
$$N^2 \epsilon \prod_{p\in S} d_p$$
for all $N$ sufficiently large.
\end{lem}

\begin{proof}
By conditions (a)--(c), the intersection of $ax+by+cz=0$ with the cube $[0,N]^3$ is a non-trivial polygonal region which up to homothety is independent of $N$.
The intersection of $ax+by+cz=0$ with $(1+D\Z )^3$ is the translate of a $2$-dimensional lattice.  If $\Lambda$ is a lattice and $R$ is a polygonal region, then
\begin{equation}
\label{coarea}
|\Lambda\cap (v+tR)| = \frac{\mathrm{Area}(R)}{\mathrm{Coarea}(\Lambda)}t^2 + O(t).
\end{equation}
Thus, the number of solutions of (\ref{abc}) in $x,y,z\in (1+D\Z )\cap [0,N]$ is of the form $AN^2+O(N)$.
By condition (d), for each $p\in S$, the conditions $p|x$, $p|y$, and $p|z$ each define a sublattice of $\Lambda$ of index $p$, so the subset $\Lambda_p$
of $\Lambda$ satisfying the condition $p\nmid xyz$
is the union of $p^2\alpha_p$ cosets of $p\Lambda$, where $\alpha_p > 1-3/p$.  By condition (f), $\alpha_2 > 0$ if $2\in S$.

Thus, $\bigcap_{p\in S} \Lambda_p$ is the union of $\prod_{p\in S} p^2\alpha_p$
cosets of $(\prod_{p\in S} p )\Lambda$.  The lemma now follows from (\ref{coarea}). 
\end{proof}

Let $X$, $Y$, and $Z$ be finite sets of integers.  The number of solutions of (\ref{abc}) with $x\in X$, $y\in Y$, and
$z\in Z$ can be written
\begin{equation}
\label{circle}
\int_0^1 \sum_{x\in X} e(axt) \sum_{y\in Y}  e(byt) \sum_{z\in Z} e(czt)\,dt
\end{equation}
where $e(t) := e^{2\pi i t}$.  

\begin{lem}
\label{ineq}
If $|\alpha_x|=|\beta_y|=|\gamma_z|=1$ for all $x,y,z$, then
\begin{align*}
\bigm|\int_0^1\sum_{x\in X} \alpha_x e(axt) &\sum_{y\in Y}  \beta_y e(byt) \sum_{z\in Z} \gamma_z e(czt) \bigm| \\
&\le \sup_t\bigm| \sum_{x\in X} \alpha_x e(axt)\bigm| |Y|^{1/2} |Z|^{1/2}.
\end{align*}
\end{lem}

\begin{proof}
By H\"older and Cauchy-Schwartz,
\begin{align*}
\bigm|\int_0^1\sum_{x\in X} \alpha_x e(axt) &\sum_{y\in Y}  \beta_y e(byt) \sum_{z\in Z} \gamma_z e(czt) \bigm| \\
&\le \Vert  \sum_{x\in X} \alpha_x e(axt)\Vert_\infty\; \Vert \!\sum_{y\in X} \beta_y e(byt)\Vert_2 \;\Vert \!\sum_{z\in X} \gamma_z e(czt)\Vert_2 \\
	&= \sup_{t\in [0,1]} \bigm|\sum_{x\in X} \alpha_x e(axt)\bigm| |Y|^{1/2} |Z|^{1/2}.
\end{align*}
\end{proof}

\begin{cor}
\label{change}
If $\delta > 0$, $X'\subset X$  has at least $(1-\delta )|X|$ elements, and $|\alpha_x|=|\beta_x|=|\gamma_x|=1$ for all $x\in X$, then
\begin{align*}
\Bigm|\int_0^1\sum_{x\in X} &\alpha_x e(axt) \sum_{y\in X}  \beta_y e(byt) \sum_{z\in X} \gamma_z e(czt)\,dt \\
& - \int_0^1\sum_{x\in X'} \alpha_x e(axt) \sum_{y\in X'}  \beta_y e(byt) \sum_{z\in X'} \gamma_z e(czt)\,dt\Bigm| \le 3\delta |X|^2.
\end{align*}

\end{cor}

Regarding the characters $f\in Q^*$ as functions on $\Q^\times$ and therefore on $X$, we can
write
\begin{equation}
\label{Qstar}
\sum_{x\in X\cap G} e(axt) = \frac 1{|Q|}\sum_{f \in Q^*}  \sum_{x\in X} f(x) e(axt),
\end{equation}
and likewise for $ \sum_{y\in X\cap G}  e(byt)$ and $\sum_{z\in X\cap G} e(czt)$.

Every complex character $\chi\colon \Q^\times/G\to U(1)$ defines a homomorphism $\Q^\times\to U(1)$
and hence a strictly multiplicative function on $\N$.  For each such function $f$ there is at most one pair $(\psi,t)$
consisting of a primitive Dirichlet character $\psi$ and a real number $t$ such that
\begin{equation}
\label{pretend}
\sum_p \frac{1-\mathrm{Re} (f(p)\bar\psi(p)p^{-it})}p < \infty,
\end{equation}
where the sum is taken over rational primes.  Following terminology of Granville and Soundararajan \cite{GS}, we will say that $f$ is \emph{pretentious} if such
a pair exists.

If $f$ takes values in  a finite subgroup of $U(1)$ (as in our case, where $f$ arises from a homomorphism $Q\to U(1)$),
and if $(\psi,t)$ satisfies (\ref{pretend}), then $t=0$.  By a theorem of Hal\'asz \cite[III.4~Theorem~4]{Tenenbaum}, for any multiplicative function $f$
which takes values in the unit disk,
\begin{equation}
\label{small-mean}
\sum_{n=1}^N f(n) = o(N)
\end{equation}
unless $f$ satisfies (\ref{pretend}) for some $t$ with $\psi=1$.
In our setting, this means (\ref{small-mean}) holds unless $f(p)=1$ outside a set $\PP_f$ of primes with
$$\sum_{p\in\PP_f} \frac 1p< \infty.$$

We denote by $Q_{\pre}^*$ the set of pretentious elements of $Q^*$.
For each $f\in Q_{\pre}^*$ there exists a unique primitive Dirichlet character $\psi$ such that 
$f$ satisfies (\ref{pretend}) with $t=0$.
We define $\PP_G$ to be the union of all the sets $\PP_{f\psi^{-1}}$  where $f\in Q^*_{\pre}$
and $\psi$ is the primitive character associated to $f$.  Again,
$$\sum_{p\in\PP_G} \frac 1p< \infty.$$
We define $D := D_G$ to be the least common multiple of the conductors of all characters $\psi$ associated with $f\in Q^*_{\pre}$.

For  $h\colon\N\to \C$,  $\alpha\in\R$, and $n\in \N$, we define
$$S_{h,n}(\alpha) := \sum_{x=1}^n e(\alpha x) h(x).$$
\begin{lem}
\label{dont-pretend}
Let $f\colon \N\to \C$ be the restriction of a homomorphism $\Q^\times\to U(1)$ with finite image, $g\colon \Z\to \C$ a periodic function, and $\alpha\in \R$.
If $f$ is not pretentious, then
$$S_{fg,n}(\alpha) = o(n).$$
\end{lem}

\begin{proof}
We claim that for all $\epsilon > 0$, there exists $m$ such that for all $n$  and all fractions $\beta = r/s$ in lowest terms
with $m < s < n/m$, we have
\begin{equation}
\label{rational}
|S_{fg,n}(\beta)| \le \epsilon n.
\end{equation}
Indeed, if $g(x)$ is periodic with period $D$, it can be written as a linear combination of $e(\gamma x)$, $\gamma\in D^{-1}\Z$.
The denominator of $\beta+\gamma$, written as a fraction in lowest terms, lies in $(m/D, Dn/m)$.
By \cite[Theorem~1]{MV}, this implies (\ref{rational}) if $m/D$ is sufficiently large.

If $\beta = r/s$ with $s\le m$, then $S_{fg,\beta}$
is a linear combination of sums of the form  $S_{f,\beta+\gamma}$,
where there are only finitely many possibilities for $\beta+\gamma$ (mod $1$).
For each possibility, $e((\beta+\gamma) x)$ is periodic of some period $k$
and can therefore be written as a linear combination of (not necessarily primitive) (mod $k$) Dirichlet characters.
By (\ref{small-mean}), 
$$S_{f\chi,1}(n) = o(n),$$
so for $n$ sufficiently large,  we have
\begin{equation}
\label{small-den}
|S_{fg,n}(\beta)| \le \frac{\epsilon n}m.
\end{equation}

To deal with  $\alpha\not\in\Q$, we follow \cite[\S6]{MV}.  For each $\alpha$, we choose the rational value $\beta=r/s$ with  $s< n/m$ which is closest to $\alpha$.
Thus, 
$$|\alpha-\beta|\le \frac{m}{ns}.$$
Summing by parts, we have
\begin{align*}
S_{fg,n}(\alpha)&= \sum_{x=1}^n e((\alpha-\beta)x)e(\beta x) f(x)g(x) \\
&= e((\alpha-\beta)n)S_{fg,n}(\beta)+
\sum_{y=1}^{n-1} e((\alpha-\beta)y)(1-e(\alpha-\beta))S_{fg,y}(\beta).
\end{align*}
If $s\ge m$, by (\ref{rational}),
\begin{align*}
|S_{fg,n}(\alpha)|&\le |S_{fg,n}(\beta)| + |\alpha-\beta|\sum_{1\le y\le n/m} |S_{fg,y}(\beta)| + |\alpha-\beta|\sum_{n/m<y\le n} |S_{fg,y}(\beta)| \\
&\le
\epsilon n + \frac1n \bigl(\frac nm\bigr)^2 + \frac 1n n^2\epsilon \le  (\frac 1{m^2}+2\epsilon)n.
\end{align*}
If $s < m$, by (\ref{small-den}),
\begin{align*}
|S_{fg,n}(\alpha)|&\le |S_{fg,n}(\beta)| + |\alpha-\beta|\sum_{1\le y\le n/m} |S_{fg,y}(\beta)| + |\alpha-\beta|\sum_{n/m<y\le n} |S_{fg,y}(\beta)| \\
&\le
\epsilon n+ \frac mn \bigl(\frac nm\bigr)^2 + \frac mn \frac{n^2\epsilon }{m} \le (\frac 1{m}+2\epsilon)n.
\end{align*}
Either way, sending $\epsilon\to 0$ and $m\to \infty$, we get the lemma.
\end{proof}

We can now prove Theorem~\ref{Linear}.
\begin{proof}
Applying Lemma~\ref{devissage} with $D=D_G$, we may assume $a,b,c$ satisfy conditions (a)--(f).
Given $\delta > 0$, let $T(\delta)$ denote the smallest integer such that 
$$\sum_{p\in \PP_G\cap [T(\delta),\infty)} \frac 1p < \delta.$$
Let $\X$  consist of all integers congruent to $1$ (mod $D$) and not divisible by any prime $p\in \PP_G\cap [2,T(\delta)]$.
Let $\X'$ denote the set of elements of $\X$ divisible by no prime in $\PP_G$.  Let $X_N := \X\cap [1,N]$ and $X'_N := \X'\cap [1,N]$.
By construction, 
$$|(X_N\cap G)\setminus (X'_n\cap G)| \le |X_N\setminus X'_N| < \delta N$$ 
for $N$ sufficiently large.  Moreover, 
$$f(x)=g(y)=h(z)=1$$
for all $f,g,h\in Q^*_{\pre}$ and $x,y,z\in X'_N$.

Let $\Sigma(X)$ denote the number of solutions of $ax+by+cz=0$ with $x,y,z\in X$.
By (\ref{circle}) and (\ref{Qstar}), $\Sigma(X_N\cap G)$ is given by
\begin{equation}
\label{XNG}
|Q|^{-3}\sum_{f,g,h\in Q^*}\int_0^1 \bigl( \sum_{x\in X_N} f(x)e(axt)\bigr) \bigl( \sum_{y\in X_N} g(y)e(byt)\bigr)\bigl( \sum_{z\in X_N} h(z)e(czt)\bigr)\,dt.
\end{equation}
By Lemma~\ref{ineq} and Lemma~\ref{dont-pretend}, if $f$ is not pretentious, the summand is $o(N^2)$.  The same is true if $g$ or $h$ is not pretentious.

By construction, for $f,g,h\in Q^*_{\pre}$, we have 
$f(x)=g(y)=h(z)=1$ for all $x,y,z\in X'_N$, so by (\ref{circle}),
$$\Sigma(X'_N) = \int_0^1 \bigl( \sum_{x\in X'_N} f(x) e(axt)\bigr) \bigl( \sum_{y\in X'_N} g(y) e(byt)\bigr)\bigl( \sum_{z\in X'_N} h(z) e(czt)\bigr)\,dt.$$
Applying Corollary~\ref{change} twice, we have
\begin{align*}
&\Bigm|\int_0^1 \bigl( \sum_{x\in X_N} f(x)e(axt)\bigr) \bigl( \sum_{y\in X_N} g(y)e(byt)\bigr)\bigl( \sum_{z\in X_N} h(z)e(czt)\bigr)\,dt - \Sigma(X_N) \Bigm|\\
\le &\Bigm|\int_0^1 \bigl( \sum_{x\in X_N} f(x)e(axt)\bigr) \bigl( \sum_{y\in X_N} g(y)e(byt)\bigr)\bigl( \sum_{z\in X_N} h(z)e(czt)\bigr)\,dt - \Sigma(X'_N)\Bigm| \\
&\qquad\qquad + |\Sigma(X'_N)-\Sigma(X_N)| \\
\le &6\delta |X_N|^2.
\end{align*}
Combining this with (\ref{XNG}), we obtain
$$\Bigm|\Sigma(X_N\cap G) - \frac{|Q^*_{\pre}|^3}{|Q|^3}\Sigma(X_N)\Bigm| = O(\delta N^2).$$

Since $\sum_{p\in \PP_G} p^{-1} < \infty$,  Lemma~\ref{product} implies
$$\limsup \frac{\Sigma(X_N)}{N^2} > 0.$$
It follows that by choosing $\delta$ sufficiently small, we can guarantee
$$\limsup \frac{\Sigma(X_N\cap G)}{N^2} > 0.$$
\end{proof}

We remark that the method of proof applies equally to the problem of solving the linear equation $ax+by+cz=0$ where
$x\in X$, $y\in Y$, and $z\in Z$, where $X$, $Y$, and $Z$ are possibly distinct finite index subgroups of $\Q^\times$.

We conclude this section with a proposition  showing that the equation (\ref{abc}) with $x,y,z\in G$
does not satisfy the naive Hasse principle.

\begin{prop}
\label{Doubles}
There exists a finite index subgroup $G$ of $\Q^\times$ and non-zero $a,b,c\in \Z$
such that $ax+by+cz=0$
has no solution in $G$ but does have a solution in the completion of $G$ in $\Q_v^\times$ for each place $v$ of $\Q$.
\end{prop}

\begin{proof}
We define
$$G := \{3^m 5^n x\mid m,n\in \Z,\ m\equiv n \pmod 4,\ x \in \Q^\times \cap \Z_3\cap \Z_5,\ x\equiv 1\pmod{15}\}.$$
Thus $G$ is of index $4\cdot \phi(15) = 32$ in $\Q^\times$.  It is
dense in $\Q_v^\times$ for $v\not\in\{3,5\}$ and for $v=p\in \{3,5\}$ its closure in $\Q_v^\times$ is 
$$G_{\{v\}} = p^{\Z}\{x\in \Z_p^\times\mid x\equiv 1\pmod p\}.$$
However, $G_{\{3,5\}}$ is not the product $G_{\{3\}}\times G_{\{5\}}$; rather, it is
$$\{(x_3,x_5)\in G_{\{3\}}\times G_{\{5\}}\mid v_3(x_3)\equiv v_5(x_5)\pmod 4\}.$$

Now, the equation
$$63x+30y+25z=0$$
has solutions in $G_{\{3\}}$ (for instance $(-5,3,9)$), but all such solutions satisfy
$$v_3(x) = v_3(y)-1 = v_3(z)-2.$$
It also has solutions in $G_{\{5\}}$ (for instance $(25, -45, -9)$), but all such solutions satisfy
$$v_5(x)=v_5(y)+1=v_5(z)+2.$$
Therefore, there are no solutions in $G_{\{3,5\}}$ and, a fortiori, no solutions in $G$.
\end{proof}

\section{Points on Diagonal Curves}
This section gives a proof of Theorem~\ref{Hasse-Minkowski}.  
It is easy to see that $G_K$ finitely generated implies $K^\times/(K^\times)^n$ finite (see, e.g., \cite{IL0}).
We begin by proving Theorem~\ref{Infinity-Zero}.

\begin{proof}
Suppose $ax^n+by^n+cz^n=0$ has a non-trivial solution $(\alpha,\beta,\gamma)\in K$.
Replacing $a,b,c$ by $a' := a\alpha^n,b' := b\beta^n,c' := c\gamma^n$
respectively, it suffices to prove that the projective curve $X': a'x^n+b'y^n+c'z^n=0$
has infinitely many points in $K$ such that $x\neq 0$, $y\neq 0$, and $z\neq 0$.
Since there are only finitely many points of $X'$ for which any of the coordinates is zero, it suffices to prove $X'(K)$ is
infinite.  The advantage of $X'$ over $X$ is that $a'+b'+c'=0$.  
Let $E\subset K$ be a number field containing $a',b',c'$.  As $E$ is infinite, we can find pairwise distinct
$p,q,r\in E^\times$ such that $a'p+b'q+c'r=0$ and an infinite sequence $h_1,h_2,\ldots\in E$ such that
all finite linear combinations of the $h_i$ with coefficients in $\{p,q,r\}$ are distinct from one another.
For each positive integer $k$, the map $f_k\colon \{p,q,r\}^k\to E$ defined by
$$f_k(x_1,\ldots,x_k) = h_1 x_1+\cdots+h_k x_k$$
is injective and takes only non-zero values.

Let $H := (K^\times)^n\cap E^\times$.
Let $m$ denote the index of $H$ in $E^\times$, which is finite.
For every positive integer $k$ the coset decomposition of $E^\times$ induces via $f_k$ a partition of $\{p,q,r\}^n$
into $m$ subsets.  By the Hales-Jewett theorem, if $k$ is sufficiently large, there exist $k$ functions $g_1,\ldots,g_k\colon \{1,2,3\}\to \{p,q,r\}$ such that for each $i$, either $g_i$ is constant or 
$$(g_i(1),g_i(2),g_i(3)) = (p,q,r),$$
and the three terms 
$$f_k(g_1(j),\ldots,g_k(j)),\ j=1,2,3,$$ 
lie in the same part of the partition.
If $I\subset \{1,\ldots,k\}$ denotes the set of indices $i$ for which $g_i$ is constant, we set
$$A = \sum_{i\in I} g_i(1)h_i,\ B = \sum_{i\in \{1,\ldots,n\}\setminus I} h_i,$$
and then $A+Bp, A+Bq, A+Br$ all belong to the same part of the partition, i.e., to the same coset of $H$.
If $C$ belongs to the inverse coset, then
$$(C(A+Bp),C(A+Bq),C(A+br))\in (E^\times)^n\times(E^\times)^n\times(E^\times)^n.$$
Thus,
$$((C(A+Bp))^{1/n},(C(A+Bq))^{1/n},(C(A+br))^{1/n})$$
lies on $X'(E)\subset X'(K)$.

\end{proof}

Now we prove Theorem~\ref{Hasse-Minkowski}.

\begin{proof}
By Theorem~\ref{Infinity-Zero} it suffices to prove that (1)$\Leftrightarrow$(2).  
For $\Q_v$ any completion of $\Q$ (i.e., $\R$ or $\Q_p$ for some $p$), we fix an algebraic closure of $\Qb_v$.  The algebraic closure $\Q^{\mathrm{cl},v}$ of $\Q$ in $\Qb_v$ is (non-canonically) isomorphic to $\Qb$.  Fixing an isomorphism $i_v\colon \Qb\to \Q^{\mathrm{cl},v}$, the restriction map defines an injective homomorphism
$G_{\Q_v}\to \Gal(\Q^{\mathrm{cl},v}/\Q)$ and via $i_v$ we obtain an injection $j_v\colon G_{\Q_v}\to G_{\Q}$.  As a topological group, $G_{\Q_v}$ is finitely generated; this is trivial if $v$ is archimedean and well-known (see, e.g., \cite{D,J,JW,W}) in the non-archimedean case.  
The invariant field $K_v$ of $\Qb$ by $j_v(G_{\Q_v})$ is isomorphic via $i_v$ to a subfield of $\Q_v$,
so (2) implies that $X(K_v)$, and therefore $X(\Q_v)$, is non-empty.

For the implication (1)$\Rightarrow$(2), we 
define $G = \Q^\times\cap (K^\times)^n$, so $G$ is of finite index in $\Q^\times$.
We apply Theorem~\ref{Linear} to $G$.  
In particular, $G\supset (\Q^\times)^n$, so by weak approximation,
for any finite set $S$ of places $v$, the closure $G_S$ of $G$ in $\Q_S^\times$ contains
$$\prod_{v\in S}(\Q_v^\times)^n.$$
In particular, if $X(\Q_v)$ has a point $(x_v:y_v:z_v)$ for each $v$, then $au+bv+cw=0$ has a solution in $G_S$ for all $S$
and therefore in $\Q$ itself, namely $u_v=x_v^n,v_v=y_v^n, w_v=z_v^n$.
\end{proof}

\begin{cor}
If $X$ is a diagonal curve, then $X(K)$ is infinite for all $K\subset \Qb$ with $G_K$ finitely generated if and only if 
$X(\A_{\Q}) \neq \emptyset$, where $\A_{\Q}$ denotes the ring of adeles.
\end{cor}

\begin{proof}
The only additional point to check is that for any $a,b,c\in \Q^\times$, there exists a finite set $S$ of places of $\Q$, including $\infty$, such that
$X$ has a point over $\Z_p$ for all $p\not\in S$.  If $p$ is sufficiently large, $a$, $b$, and $c$ are $p$-adic units, so $X$ has good reduction (mod $p$),
and the reduction is a curve of genus $\frac{(n-1)(n-2)}2$.  If $p > (n-1)^2(n-2)^2$, the Weil bound implies that $X$ has at least one points over $\F_p$,
and Hensel's lemma implies that any such point lifts to a $\Z_p$-point.

\end{proof}

\begin{quest}
Is it always true that  for $X$  a non-singular curve over a number field $E$, there exists an $\A_E$-point on $X$ if and only if
for all $K\subset \Qb$ with $G_K$ finitely generated, $X(K)$ is infinite?
\end{quest}

The circle method offers the hope of giving an affirmative answer to this question for some non-diagonal curves.  We hope to treat this matter in a subsequent paper.

\end{document}